\documentclass{daj}
\usepackage{amsmath,amsfonts,amsbsy,amsgen,amscd,mathrsfs,amssymb,amsthm,bm}
\usepackage{enumerate,mathtools}

\usepackage{tikz}
\usetikzlibrary{shadings}

\newtheorem{theorem}{Theorem}
\newtheorem{lemma}[theorem]{Lemma}
\newtheorem{corollary}[theorem]{Corollary}
\newtheorem*{defn}{Definition}
\newtheorem*{prop*}{Proposition}

\newtheorem*{conj*}{Conjecture}

\newtheorem*{fact*}{Fact}
\newtheorem{prop}{Proposition}
\newtheorem*{ex*}{Example}

{\theoremstyle{remark}
\newtheorem{rem}{Remark}}

\theoremstyle{definition}

\makeatletter
\g@addto@macro\bfseries{\boldmath}
\makeatother

\newcommand*{\claimproofname}{Proof of claim}

\DeclareMathOperator*{\E}{\mathbb{E}}

\newcommand{\eps}{\varepsilon}
\newcommand{\om}{\omega}
\newcommand{\Om}{\Omega}
\newcommand{\grpord}{K}
\newcommand{\supp}{\textnormal{supp}}
\DeclareMathOperator*{\argmax}{arg\,max}

\newcommand{\Gtrsim}[1]{\gtrsim_{{\textstyle\mathstrut}{#1}}}

\newcommand{\Lesssim}[1]{\lesssim_{{\textstyle\mathstrut}{#1}}}

\DeclareMathSymbol{\shortminus}{\mathbin}{AMSa}{"39}

\numberwithin{equation}{section} 
\numberwithin{figure}{section}
\numberwithin{table}{section}

\newcommand{\R}{\mathbf{R}}
\newcommand{\C}{\mathbf{C}}

\newcommand{\T}{\mathbf{T}}
\newcommand{\D}{\mathbf{D}}

\newcommand{\Z}{\mathbf{Z}}

\newcommand{\iu}{\mathbf{i}}

\newcommand{\underbracedmatrix}[2]{%
	\left(\;
	\smash[b]{\underbrace{
			\begin{matrix}#1\end{matrix}
		}_{#2}}
	\;\right)
	\vphantom{\underbrace{\begin{matrix}#1\end{matrix}}_{#2}}
}
\dajAUTHORdetails{%
  title = {A dimension-free discrete Remez-type inequality on the polytorus}, 
  author = {Joseph Slote, Alexander Volberg, and Haonan Zhang},
  plaintextauthor = {Joseph Slote, Alexander Volberg, and Haonan Zhang},
  keywords = {Remez inequality, Berstein-type discretization inequality, Bohnenblust--Hille inequality},
}   

\dajEDITORdetails{%
   year={2025},
   number={4},
   received={11 November 2023},   
   published={15 May 2025},  
   doi={10.19086/da.137968},       
}   

\begin{document}

\begin{frontmatter}[classification=text]


\author[joe]{Joseph Slote\thanks{Supported by Chris Umans' Simons Institute Investigator Grant.}}
\author[sasha]{Alexander Volberg\thanks{Supported by NSF DMS-1900286, DMS-2154402 and by Hausdorff Center for Mathematics.}}
\author[haonan]{Haonan Zhang}

\begin{abstract}
Consider $f:\Omega^n_K \to \mathbf{C}$ a function from the $n$-fold product of multiplicative cyclic groups of order $K$.
	Any such $f$ may be extended via its Fourier expansion to an analytic polynomial on the polytorus $\mathbf{T}^n$, and the set of such polynomials coincides with the set of all analytic polynomials on $\mathbf{T}^n$ of individual degree at most $K-1$.
	
	In this setting it is natural to ask how the supremum norms of $f$ over $\mathbf{T}^n$ and over $\Omega_K^n$ compare.
	We prove the following \emph{discretization of the uniform norm} for low-degree polynomials: if $f$ has degree at most $d$ as an analytic polynomial, then $\|f\|_{\mathbf{T}^n}\leq C(d,K)\|f\|_{\Omega_K^n}$ with $C(d,K)$ independent of dimension $n$.
	As a consequence we also obtain a new proof of the Bohnenblust--Hille inequality for functions on products of cyclic groups.
	
	Key to our argument is a special class of Fourier multipliers on $\Omega_K^n$ which are $L^\infty\to L^\infty$ bounded independent of dimension when restricted to low-degree polynomials. This class includes projections onto the $k$-homogeneous parts of low-degree polynomials as well as projections of much finer granularity. 

\end{abstract}
\end{frontmatter}



\section{Introduction}
	We say an analytic polynomial $f:\T^n\to \C$ has \emph{(total) degree} at most $d$ and \emph{individual degree} at most $K-1$ if it has the (Fourier) expansion
	\begin{equation}
	\label{eq:fourier-exp}
		f(z)=\sum_{\alpha\in \{0,1,\dots, K-1\}^n:|\alpha|\le d}\widehat{f}(\alpha)z^{\alpha},\qquad \widehat{f}(\alpha)\in \C.
	\end{equation}
	We shall use the notation $|\alpha|:=\sum_j\alpha_j$ and $z^\alpha:=\prod_{j}z^{\alpha_j}$ throughout.
	Fixing $K\ge 1$ a positive integer, let $\Omega_K:=\{1,\omega_K, \dots, \omega_K^{K-1}\}$ denote the multiplicative cyclic group of order $K$, where $\omega_K:=e^{2\pi \iu/K}$.
	
	The subject of this paper is a certain \emph{discrete Remez-type inequality} (or \emph{discretization of the uniform norm}) for analytic polynomials on the polytorus $\T^n$.
	As observed in \cite{DMP}, when $\grpord=2$ (so $f$ is multi-affine) a comparison of Klimek \cite{K} entails that
	\begin{equation}
	\label{eq:k2case}
		\|f\|_{\T^n}\leq (1+\sqrt2)^d\|f\|_{\Om_2^n}\,,
	\end{equation}
	where here and throughout $\|f\|_X:=\sup_{x\in X}|f(x)|$ denotes the supremum norm.
	In the sequel we prove that the dimension-free comparison \eqref{eq:k2case} is in fact a special case of a phenomenon that holds true for analytic polynomials of any individual degree:
	
	\begin{theorem}\label{thm: remez}
	Let $d,n\ge 1, \grpord\ge 2$. Suppose $f$ is an analytic polynomial of degree at most $d$ and individual degree at most $K-1$.
	Then
\begin{equation}\label{ineq:remez cyclic}
	\|f\|_{\T^n}\leq C(d,K)\|f\|_{\Omega_K^n}
\end{equation}
for some constant $C(d,\grpord)$ depending on $d$ and $\grpord$ only.
\end{theorem}

\begin{rem}
	Individual degree is never more than total degree, so we also have more simply that for any analytic polynomial $f:\T^n\to \C$ with $\deg(f)< d,$
	\[\|f\|_{\T^n}\Lesssim{d}\|f\|_{\Omega_{d}^n}\,.\]
	The notation $A\Lesssim{d}B$ means $A\le C(d)B$ for some constant $C(d)>0$ depending on $d$ only.
\end{rem}

The key feature of \eqref{ineq:remez cyclic} is its lack of dependence on dimension.
We are not too concerned with the explicit constant $C(d,\grpord)$ here;
Theorem \ref{thm: remez} is generalized and improved in a later work \cite{BKSVZ} by Becker, Klein, and the present authors, where an explicit constant $C(d,K)= (\mathcal{O}(\log \grpord))^{2d}$ for \eqref{ineq:remez cyclic} is proved, along with similar results for much more general sampling sets than $\Om_\grpord^n$.

To some extent one may consider the present paper an important step toward these later improvements \cite{BKSVZ}.
However, the techniques in the present work are different and of independent interest.
Whereas the main technique of \cite{BKSVZ} is a probabilistic argument to establish a special polynomial interpolation formula, the present work develops a new class of Fourier multipliers which are bounded independent of dimension when applied to low-degree polynomials.

More concretely, for an $S\subset\{0,1,\ldots, K-1\}^n$ consider the $S$-part of $f$: \[f_S := \sum_{\alpha\in S}\widehat{f}(\alpha)z^\alpha\, .\]
We show that for a rich collection of $S$'s it holds that
\begin{equation}
    \label{ineq:f-mult}
    \|f_S\|_{\Om_\grpord^n}\Lesssim{d,K} \|f\|_{\Om_\grpord^n}\, ,
\end{equation}
for $f$ of degree at most $d$ and  individual degree at most $\grpord-1$.
Related results are used throughout but the class of $S$'s is studied most directly in Sections \ref{sec:boundedness} and \ref{sec:aside}.
When $K$ is prime this class has an explicit and self-contained description, following from some results in transcendental number theory.

\begin{theorem}
	\label{cor:bounded-parts-intro}
	Suppose $\grpord$ is an odd prime and let $S$ be a maximal subset of $\{0,1,\ldots, \grpord-1\}^n$ such that for all $\alpha, \beta\in S$:
	\begin{itemize}
		\item Support sizes are equal: $|\supp(\alpha)|=|\supp(\beta)|$.
		\item Degrees are equal: $|\alpha|=|\beta|$.
		\item Individual degree symmetry: there is a bijection $\pi:\supp(\alpha)\to\supp(\beta)$ such that for all $j\in\supp(\alpha)$, $\alpha_j=\beta_{\pi(j)}$ or $\alpha_j = \grpord-\beta_{\pi(j)}$.
	\end{itemize}
	Then for any $n$-variate analytic polynomial $f$ of degree at most $d$ and individual degree at most $K-1$, the $S$-part of $f$ satisfies:
	\[\|f_S\|_{\Om_\grpord^n}\Lesssim{d,\grpord}\|f\|_{\Om_\grpord^n}\,.\]
\end{theorem}
\noindent Here the support of a monomial $z^\alpha$ is defined to be $\supp(\alpha):=\{j:\alpha_j\neq 0\}$, and the support size $|\supp(\alpha)|$ refers to the cardinality of $\supp(\alpha)$.
Theorem \ref{cor:bounded-parts-intro} and related techniques do not seem to follow from the argument in \cite{BKSVZ} and can be considered as one of the main contributions of this work.

\medskip
We prove Theorem \ref{thm: remez} in Section \ref{sec:proof} and Theorem \ref{cor:bounded-parts-intro} as Corollary \ref{cor:bounded-parts} in Section \ref{sec:aside}.
Here we remark on some interpretations and consequences of \eqref{ineq:remez cyclic}.

\begin{rem}
	Theorem \ref{thm: remez} can be viewed as a \emph{generalized maximum modulus principle} since it implies the dimension-free boundedness on the entire polydisk: For any analytic polynomial $f:\D^n\to \C$ of degree at most $d$ and individual degree $\grpord-1$ we have
\begin{equation*}
	\|f\|_{\D^n}\Lesssim{d,\grpord} \|f\|_{\Omega_\grpord^n},
\end{equation*}
where $\D:=\{z\in \C:|z|\le 1\}$ is the closed unit disk.

Moreover, let $f$ be as in Theorem \ref{thm: remez}. By a standard Cauchy estimate, we have $\|f_k\|_{\T^n}\le \|f\|_{\T^n}$ with $f_k$ being the $k$-homogeneous part of $f$. This, together with \eqref{ineq:remez cyclic}, implies 
\begin{equation}\label{ineq:generalized cauchy}
\|f_k\|_{\Omega_K^n}\le 	\|f_k\|_{\T^n}\le \|f\|_{\T^n}\le  C(d,K)\|f\|_{\Omega_K^n}.
\end{equation}
So a Cauchy-type estimate holds for $\Omega_K^n$ as well.
\end{rem}

\subsection{Functions on \texorpdfstring{$\Om_\grpord^n$}{Ωₖⁿ} and the Cyclic-group Bohnenblust--Hille Inequality}

A central application of Theorem \ref{thm: remez} is the study of functions $f:\Om_\grpord^n\to \C$ on products of cyclic groups.
Any such $f$ may be extended via its Fourier expansion to an analytic polynomial on $\T^n$ with individual degree at most $\grpord-1$.
In this way Theorem \ref{thm: remez} implies the $L^\infty\to L^\infty$-boundedness of this extension map when $f$ is of bounded total degree.

As an immediate corollary we obtain a Bohnenblust--Hille-type inequality for functions on $\Om_\grpord^n$.
The original Bohnenblust--Hille (BH) inequality \cite{BH} states
	\begin{equation}
	\label{eq:BH}
		{\|\widehat f\|_{\frac{2d}{d+1}}\Lesssim{d}\|f\|_{\T^n}}
	\end{equation}
	for analytic polynomials $f:\T^n\to \C$ of degree at most $d$.
Here $\|\widehat{f}\|_p$ denotes the $\ell^p$-norm of the Fourier coefficients of $f$; that is,
\[\|\widehat{f}\|_{p}:=\left(\textstyle\sum_{\alpha}|\widehat{f}(\alpha)|^p\right)^{1/p},\]
for $f$ expanded as in \eqref{eq:fourier-exp} (with $\grpord=\infty$).
Again one key property of the BH inequality is its dimension-freeness.
Combining \eqref{eq:BH} with \eqref{ineq:remez cyclic} we obtain:

\begin{corollary}[Cyclic-group Bohnenblust--Hille]
Let $d,n\ge 1, \grpord\ge 2$. Let $f:\Om_\grpord^n\to\C$ with $\deg(f)\leq d$. 
Then
	\[\|\widehat f\|_{\frac{2d}{d+1}}\Lesssim{d,\grpord}\|f\|_{\Om_\grpord^n}\,.\]
\end{corollary}
The cyclic-group Bohnenblust--Hille inequality was originally proved in \cite{SVZbh} with an argument avoiding Theorem \ref{thm: remez}.
Even though the cyclic-group BH inequality for $2< \grpord<\infty$ interpolates between the now well-understood polytorus ($\grpord=\infty$) and hypercube ($\grpord=2$) cases of the BH inequality (\cite{BH} and \cite{blei,DMP} respectively), the $2< \grpord<\infty$ case does not appear to follow from the ``standard recipe'' for BH inequalities---and so a new fact, such as Theorem \ref{thm: remez}, is needed.
See \cite{SVZbh} for an explanation of the challenges involved, as well as for extensions to BH inequalities for discrete quantum systems in the spirit of \cite{VZ22}.

\subsection{Discrete Remez-type inequalities and discretizations of the uniform norm}
\label{sec:remez}

Theorem \ref{thm: remez} can be understood as a dimension-free refinement of (a special case of) existing \emph{discrete Remez-type inequalities}.
It can also be considered as a \emph{discretization of the uniform norm} (also known as a \emph{Bernstein-type discretization inequality}).
We discuss these connections in order.

\subsubsection*{Remez-type inequalities in many dimensions}
Consider $J$ a finite interval in $\R$ and a subset $E\subset J$ with positive Lebesgue measure $\mu(E)>0$.
Let $f:\R\to \R$ be a real polynomial of degree at most $d$.
The classical Remez inequality \cite{R} states that 
\begin{equation}\label{ineq:remez}
	\max_{x\in J}|f(x)|\le \left(\frac{4\mu(J)}{\mu(E)}\right)^d \max_{x\in E}|f(x)|.
\end{equation}

Despite a large literature extending \eqref{ineq:remez}, we are not aware of any direct multi-dimensional generalizations that are dimension-free.
Multi-dimensional versions of the Remez inequality are considered in the papers of Brudnyi and Ganzburg \cite{BG}, Ganzburg \cite{G}, Kro\'o and Schmidt \cite{KS} but they are not at all dimension-free: it is instructive to take a look at inequality (23) in \cite{KS} and see how the estimates blow up with dimension (called $m$ in \cite{KS}).
If one abandons the $L^\infty$ norm on the left-hand side of \eqref{ineq:remez} then something can be said; there are distribution function inequalities for volumes of level sets of polynomials that are dimension-free, see \cite{Fr,NSV,NSV1}.
But those are distribution function estimates, not $L^\infty$ estimates. Some other related results include Nazarov's extension \cite{Nazarov} of Tur\'an's inequality \cite{T}, as well as more generalizations \cite{FM,FY}.  

The lack of a dimension-free multi-dimensional Remez inequality of the form \eqref{ineq:remez} is not surprising: there is no hope for such an inequality phrased in terms of $\mu(E)$ for any positive-measure $E\subseteq J$.
This can already be seen when $J$ is a unit ball in $\R^n$ and $f_n(x)=1-\sum_{j=1}^n x_j^2$.
For large $n$, most of the volume of the ball is concentrated in a neighborhood of the unit sphere where $f_n$ is very small.

However, this observation does not preclude the existence of \emph{certain} sets $E$ giving multi-dimensional analogues of \eqref{ineq:remez} that are dimension-free.
Indeed, Lundin \cite{L}, and later Aron--Beauzamy--Enflo \cite{ABE} and Klimek \cite{K}, show this is possible in certain cases of $(J,E)$ with convex $E$, such as for bounded-degree polynomials over the polydisk $J=\D^n$ and the real cube $E=[-1,1]^n$.
As an explicit example, with the prevailing notation, Klimek \cite{K} showed that for $n$-variate analytic polynomials of degree $d$, we have the comparison $\|f\|_{\D^n}\leq (1+\sqrt{2})^d\|f\|_{[-1,1]^n}$.

On the other hand, it was not at all clear when dimension-free Remez inequalities should exist in non-convex settings like $J=\T^n$ and $E\subset \T^n$.
The arguments in \cite{L,ABE,K} make essential use of the convexity of the testing set $E$ and do not seem to suitably generalize.\footnote{If one considers only \emph{multi-affine} (or individual degree at most 1) polynomials $f$, then $\|f\|_{[-1,1]^n}=\|f\|_{\Om_2^n}$, and by Klimek \cite{K} one obtains \eqref{eq:k2case}; that is, $\|f\|_{\T^n}\leq {(1+\sqrt{2})^d}\|f\|_{\Om_2^n}$.
This was observed in \cite{DMP}.
But this line of argument does not appear to extend beyond the class of multi-affine polynomials.}
In comparison, for our application to functions on products of cyclic groups $f:\Om_\grpord^n\to\C$, we have no choice but to use the non-convex grid $\Om_\grpord^n$ as our $E$.%

That our $E$ is discrete and indeed finite is another interesting feature.
Remez-type estimates with discrete $E$ were known before; notably, Yomdin \cite{Y} (see also \cite{BY}) identifies a geometric invariant which directly replaces the Lebesgue measure in \eqref{ineq:remez} and is positive for certain finite sets $E$---though the comparison is not dimension-free.%

It is natural to ask for what other (discrete) sets $E\subset \mathbf{D}^n$ a dimension-free comparison might hold.
In \cite{BKSVZ} our Theorem \ref{thm: remez} is extended to a much larger class of testing sets, but we are far from a full characterization of such $E$.

\subsubsection*{Discretizations of the uniform norm}

In one dimension ($n=1$) the inequality \eqref{ineq:remez cyclic} is a classical theorem of Bernstein and generalizations are known as \emph{discretizations of the uniform norm} or \emph{Bernstein-type discretization inequalities} (see \cite{bernstein31,bernstein32} and \cite[Chapter X, Theorem (7.28)]{Zygmund}).
We refer to surveys \cite{DPTT19,KKLT22} and references therein for more historical background about norm discretizations.

In the high-dimensional case, Theorem \ref{thm: remez} can be understood as a Bernstein-type discretization inequality for bounded-degree multivariate polynomials in many dimensions $n$.
Such inequalities have been the subject of much study in approximation theory.
However, existing high-dimensional Bernstein-type estimates do not seem to apply to our situation when the sampling set is the fixed discrete torus $\Omega_\grpord^n$. We refer to \cite{BKSVZ} for more discussion and comparison of our work with known literature. 

\section{The Proof}
\label{sec:proof}

As the $K=2$ case was known, we will focus on proving the $K\geq 3$ case.
This proof uses some ideas and techniques from \cite{SVZ,SVZbh}. Recall that we need to prove 
\begin{equation*}
	\|f\|_{\T^n}\Lesssim{d,\grpord}\|f\|_{\Omega_\grpord^n}
\end{equation*}
for all analytic polynomials $f:\T^n\to \C$ of degree at most $d$ and individual degree at most $\grpord-1$.
For this, we divide the proof into two steps:
\begin{align*}
		\text{Step 1.}\quad \|f\|_{\T^n}\hspace{0.4em}&\Lesssim{d,\grpord}\|f\|_{\Omega_{2\grpord}^n}, \quad \text{and}\\[0.5em]
	\text{Step 2.} \quad \|f\|_{\Omega_{2\grpord}^n}&\Lesssim{d,\grpord}\|f\|_{\Omega_{\grpord}^n}\,.
\end{align*}

\subsection{Step 1}

\begin{prop}[Torus bounded by \unboldmath$\Omega_{2\grpord}$]
			\label{prop:Om2K-boundedness}
	Let $d,n\ge 1, K\ge 3$. Let $f:\T^n\to \C$ be an analytic polynomial of degree at most $d$ and individual degree at most $\grpord-1$.
	Then 
	\begin{equation*}
				\|f\|_{\T^n}\le C_\grpord^d\|f\|_{\Omega_{2\grpord}^n}\,,
	\end{equation*}
where $C_\grpord\ge 1$ is a universal constant depending on $\grpord$ only.
\end{prop}

To prove this proposition, we need the following lemma.

\begin{lemma}
	\label{lem:probability measure}
	Fix $\grpord\ge 3$.
	There exists $\eps=\eps(\grpord)\in (0,1)$ such that, for all $z\in\C$ with $|z|\le \eps$, one can find a probability measure $\mu_z$ on $\Om_{2\grpord}$ such that
	\begin{equation}\label{eq:complex equations to solve}
		z^m=\E_{\xi\sim\mu_z}\xi^m, \qquad \forall \quad 0\le m\le \grpord-1\,.
	\end{equation}
\end{lemma}

\begin{proof}
	Put $\theta=2\pi/2\grpord=\pi/\grpord$ and $\omega=\omega_{2\grpord}=e^{\iu\theta}$.
	Fix a $z\in \C$. Finding a probability measure $\mu_z$ on $\Om_{2\grpord}$ satisfying \eqref{eq:complex equations to solve} is equivalent to solving 
	\begin{equation}\label{eq:real equations to solve}
		\begin{cases}
			\sum_{k=0}^{2\grpord-1}p_k=1&\\
			\sum_{k=0}^{2\grpord-1}p_k\cos(km\theta)=\Re z^m& 1\le m\le \grpord-1\\						
			\sum_{k=0}^{2\grpord-1}p_k\sin(km\theta)=\Im z^m& 1\le m\le \grpord-1
		\end{cases}
	\end{equation}
	with non-negative $p_k=\mu_z(\{\omega^k\})$ for $k=0,1,\ldots, 2\grpord-1$. Note that the $p_k$'s are non-negative and thus real.
	
	For this, it is sufficient to find a solution $\vec{p}=\vec{p}_z$ to $D_\grpord \vec p=\vec v_z$ with each entry of $\vec p=(p_0,\ldots, p_{2\grpord-1})^\top$ being non-negative. Here $D_\grpord$ is a $2\grpord\times 2\grpord$ real matrix given by
	\begin{equation*}
		D_\grpord=
		\begin{bmatrix}
			1&1 & 1 &\cdots & 1\\
			1&  \cos(\theta) &\cos(2\theta)&\cdots &\cos((2\grpord-1)\theta)\\
			\vdots & \vdots & \vdots &\vdots \\
			1&\cos(\grpord\theta) &\cos(2\grpord\theta)&\cdots &\cos((2\grpord-1)\grpord\theta)\\[0.3em]
			1&\sin(\theta) &\sin(2\theta)&\cdots &\sin((2\grpord-1)\theta)\\
			\vdots&\vdots & \vdots & \vdots &\vdots \\
			1&\sin((\grpord-1)\theta) &\sin(2(\grpord-1)\theta)&\cdots &\sin((2\grpord-1)(\grpord-1)\theta)\\
		\end{bmatrix},
	\end{equation*}
	and $\vec v_z=(1,\Re z, \dots, \Re z^{\grpord-1},\Re z^{\grpord},\Im z,\dots, \Im z^{\grpord-1})^\top\in \R^{2\grpord}$. Note that \eqref{eq:real equations to solve} does not require the $(\grpord+1)$-th row
	\begin{equation}\label{eq:K+1 row}
		(1,\cos(\grpord\theta),\cos(2\grpord\theta),\dots, \cos((2\grpord-1)\grpord\theta))
	\end{equation}
	of $D_\grpord$.

	The matrix $D_\grpord$ is non-singular. To see this, take any $$\vec{x}=(x_0,x_1,\dots, x_{2\grpord-1})^\top\in \R^{2\grpord}$$ 
	such that $D_\grpord \vec{x}=\vec{0}$. Then 
	\begin{equation}\label{eq:zero}
		\sum_{k=0}^{2\grpord-1}(\omega^k)^m x_k=0,\qquad 0\le m\le \grpord.
	\end{equation}
	This is immediate for $0\le m\le \grpord-1$ by definition, and $m=\grpord$ case follows from the ``additional" row \eqref{eq:K+1 row} together with the fact that $\sin(k\grpord\theta)=0,0\le k\le 2\grpord-1.$
	Conjugating \eqref{eq:zero}, we get
	\begin{equation*}
		\sum_{k=0}^{2\grpord-1}(\omega^k)^m x_k=0,\qquad \grpord\le m\le 2\grpord.
	\end{equation*}
	Altogether, we have 
	\begin{equation*}
		\sum_{k=0}^{2\grpord-1}(\omega^k)^m x_k=0,\qquad 0\le m\le 2\grpord-1,
	\end{equation*}
	that is, $V\vec{x}=\vec{0}$, where $V=V_\grpord=[\omega^{jk}]_{0\le j,k\le 2\grpord-1}$ is a $2\grpord\times 2\grpord$ Vandermonde matrix given by $(1,\omega,\dots, \omega^{2\grpord-1})$. 	Since $V$ has determinant
	\begin{equation*}
		\det(V)=\prod_{0\le j<k\le 2\grpord-1}(\omega^j-\omega^k)\neq 0\,,
	\end{equation*}
	we get $\vec x=\vec 0$. So $D_\grpord$ is non-singular. 
	
	Therefore, for any $z\in\C$, the solution to \eqref{eq:real equations to solve}, thus to \eqref{eq:complex equations to solve}, is given by 
	\[
	\vec p_z=\big(p_0(z),p_1(z),\dots, p_{2\grpord-1}(z)\big)=D_\grpord^{-1}\vec v_z\in \R^{2\grpord}.
	\] 
	
	Notice one more thing about the rows of $D_\grpord$. As
	\[
	\sum_{k=0}^{2\grpord-1} (\omega^k)^m =0,\qquad m=1, 2, \dots, 2\grpord-1\,,
	\]
	we have automatically that vector $\vec p_* := \big(\frac1{2\grpord}, \dots, \frac1{2\grpord}\big)\in \R^{2\grpord}$ gives
	\[
	D_\grpord \vec p_* =(1, 0, 0, \dots, 0)^T=:\vec v_*\,.
	\]
	For any $k$-by-$k$ matrix $A$ denote 
	\[\|A\|_{\infty\to\infty}:=\sup_{\vec 0\neq \vec v\in\R^{k}}\frac{\|A\vec v\|_\infty}{\|\vec v\|_\infty}.\]
	So with $\vec{p_*}:= D^{-1}_\grpord\vec{v_*}$  we have 
	\begin{align*}
		\|\vec p_z-\vec p_*\|_\infty
		&\le \|D_\grpord^{-1}\|_{\infty\to \infty}\|\vec v_z-\vec v_*\|_\infty\\
		&=\|D_\grpord^{-1}\|_{\infty\to \infty}\max\left\{\max_{1\le k\le \grpord}|\Re z^k|,\max_{1\le k\le \grpord-1}|\Im z^k|\right\}\\
		&\le \|D_\grpord^{-1}\|_{\infty\to \infty}\max\{|z|,|z|^\grpord\}.
	\end{align*}
	That is, 
	\begin{equation*}
		\max_{0\le j\le 2\grpord-1}\left|p_j(z)-\frac{1}{2\grpord}\right|
		\le \|D_\grpord^{-1}\|_{\infty\to \infty}\max\{|z|,|z|^{\grpord}\}.
	\end{equation*}
	Since $D_\grpord^{-1}\vec v_*=\vec p_*$, we have $\|D_\grpord^{-1}\|_{\infty\to \infty}\ge 2\grpord$.
	Put
	\[
	\eps_*:=\frac{1}{2\grpord\|D_\grpord^{-1}\|_{\infty\to \infty}}\in \left(0,\frac{1}{(2\grpord)^2}\right].
	\]
	Thus whenever $|z|<\eps_*<1$, we have 
	\begin{equation*}
		\max_{0\le j\le 2\grpord-1}\left|p_j(z)-\frac{1}{2\grpord}\right|
		\le |z|\|D_\grpord^{-1}\|_{\infty\to \infty}
		\le \eps_* \|D_\grpord^{-1}\|_{\infty\to \infty}\le \frac{1}{2\grpord},
	\end{equation*}
	so in particular $p_j(z)\ge 0$ for all $0\le j\le 2\grpord-1$.
\end{proof}

Now we are ready to prove Proposition \ref{prop:Om2K-boundedness}.

	\begin{proof}[Proof of Proposition \ref{prop:Om2K-boundedness}]
	Let $\eps_*$ be as in Lemma \ref{lem:probability measure}.
	With a view towards applying the lemma we begin by relating $\sup |f|$ over the polytorus to $\sup |f|$ over a scaled copy.
	Recalling that the homogeneous parts $f_k$ of $f$ are trivially bounded by $f$ over the torus: $\|f_k\|_{\T^n}\le \|f\|_{\T^n}$ (a standard Cauchy estimate). Thus we have
	\begin{align}
		\label{eq:torus-dilation}
		\|f\|_{\T^n}&\leq \sum_{k=0}^d\|f_k\|_{\T^n}\nonumber\\
		&= \sum_{k=0}^d \eps_*^{-k}\sup_{z\in\T^n}|f_k(\eps_* z)|\nonumber\\
		&\leq \sum_{k=0}^d \eps_*^{-k}\sup_{z\in\T^n}|f(\eps_* z)|\nonumber\\
		&\leq (d+1)\eps_*^{-d}\sup_{z\in\T^n}|f(\eps_* z)|\nonumber\\
		&= (d+1)\eps_*^{-d}\|f\|_{(\eps_*\T)^n}\,.
	\end{align}
	
	Let $z=(z_1,\ldots, z_n)\in (\eps_* \T)^n$.
	Then for each coordinate $j=1,2,\ldots, n$ there exists by Lemma \ref{lem:probability measure} a probability distribution $\mu_j=\mu_j(z)$ on $\Om_{2\grpord}$ for which $\E_{\xi_j\sim\mu_j}[\xi_j^k]=z_j^k$ for all $0\le k\le \grpord-1$.
	With $\mu =\mu(z):=\mu_1\times\cdots\times\mu_n$, this implies for a monomial $\xi^\alpha$ with multi-index $\alpha\in \{0,1,\dots, \grpord-1\}^n$, $\E_{\xi\sim\mu(z)}[\xi^\alpha]=z^\alpha,$ or more generally by linearity $\E_{\xi\sim\mu(z)}[f(\xi)]=f(z)$ for $z\in(\eps_*\T)^n$ and $f$ under consideration.
	So
	\begin{equation}
		\label{eq:dilated-torus-comparison}
		\sup_{z\in(\eps_* \T)^n}|f(z)|= \sup_{z\in(\eps_* \T)^n}\Big|\E_{\xi\sim\mu(z)}f(\xi)\Big|\leq  \sup_{z\in(\eps_* \T)^n}\E_{\xi\sim\mu(z)}|f(\xi)|\leq\|f\|_{\Om_{2\grpord}^n}.
	\end{equation}
	Combining observations \eqref{eq:torus-dilation} and \eqref{eq:dilated-torus-comparison} we conclude
	\begin{equation*}
		\|f\|_{\T^n}\leq (d+1)\eps_*^{-d}\|f\|_{(\eps_*\T)^n}\leq (d+1)\eps_*^{-d}\|f\|_{\Om_{2\grpord}^n}
		\leq C_\grpord^d\|f\|_{\Om_{2\grpord}^n}. \qedhere
	\end{equation*}
    The last inequality follows from the fact that $\eps_*$ depends only on $K$.

\end{proof}

\subsection{Step 2}

Now we turn to Step 2's estimate,
\begin{equation}
\label{ineq:step-2}
	\|f\|_{\Om_{2\grpord}^n}\Lesssim{d,\grpord}\|f\|_{\Om_\grpord^n}\,.
\end{equation}
We will find it useful to rephrase this question as one about the boundedness at the single point
\[f(\sqrt{\omega},\ldots, \sqrt{\omega})=:f(\sqrt{\omega})\,.\]
Here and in what follows, $\omega:=\omega_\grpord=e^{2\pi \iu/\grpord}$, and $\sqrt{\omega}$ will be used as shorthand to denote the root $e^{\pi \iu/\grpord}$.
It turns out the following proposition is enough to give \eqref{ineq:step-2}.

\begin{prop}
	\label{lem:sqrtom-boundedness}
		Let $d,n\ge 1, K\ge 3$. Let $f:\T^n\to \C$ be an analytic polynomial of degree at most $d$ and individual degree at most $\grpord-1$. Then 
	\[|f(\sqrt{\om})| \Lesssim{d,\grpord} \|f\|_{\Om_\grpord^n}\,.\]
\end{prop}

To explain why Proposition \ref{lem:sqrtom-boundedness} suffices for Step 2, let us finish the proof of Theorem \ref{thm: remez} given Proposition \ref{prop:Om2K-boundedness} and assuming Proposition \ref{lem:sqrtom-boundedness}.

\begin{proof}[Proof of Theorem \ref{thm: remez}]
	Fix a $z^*\in \argmax_{z\in\Om_{2\grpord}^n}|f(z)|$.
	Then there exist $w=(w_1,\dots, w_n)\in \Om_\grpord^n$ and $y^*\in\{1,\sqrt{\omega}\}^n$ such that
	\[w_jy^*_j = z^*_j, \quad j\in [n]\,,\]
	where $[n]:=\{1,2,\dots, n\}.$
	Define $\widetilde{f}:\T^n\to\C$ by
	\[\widetilde{f}(z) = f(w_1z_1,w_2z_2,\ldots, w_nz_n)\,.\]
	We therefore have
	\begin{align}
		\label{eq:fwidetilde-opt}
		|\widetilde{f}(y^*)|&=\|f\|_{\Om_{2\grpord}^n}\quad\text{and}\\
		\label{eq:fwidetilde-om3}\|\widetilde{f}\|_{\Om_\grpord^n}&=\|f\|_{\Om_\grpord^n}\,.
	\end{align}
	Equation \eqref{eq:fwidetilde-opt} holds
	by the definition of $y^*$, and \eqref{eq:fwidetilde-om3} holds by the group property of $\Om_\grpord$ (recall $w\in\Om_\grpord^n$).
	
	Now let $S=\{j:y^*_j=\sqrt{\omega}\}$ and $m=|S|$.
	Let $\pi:S\to[m]$ be any bijection.
	Define the ``selector'' function $s_{y^*}:\T^m\to \T^n$ coordinate-wise by
	\[\big(s_{y^*}(z)\big)_j = \begin{cases}
		y^*_j & \text{if } j\not\in S\\
		z_{\pi(j)} & \text{if } j\in S\,.
	\end{cases}\]
	Finally, define $g:\T^m\to\C$ by
	\[g(z) = \widetilde{f}(s_{y^*}(z))\,.\]
	Then we observe that $g$ is analytic with degree at most $d$ and individual degree at most $K-1$, and
	\begin{align}
		\label{eq:g-6}
		|g(\sqrt{\omega},\sqrt{\omega},\ldots, \sqrt{\omega})| = |\widetilde{f}(y^*)|\overset{\text{\eqref{eq:fwidetilde-opt}}}{=} \|f\|_{\Om_{2\grpord}^n}\\
		\label{eq:g-3}
		\|g\|_{\Om_\grpord^m}\leq \|\widetilde{f}\|_{\Om_\grpord^n}\overset{\text{\eqref{eq:fwidetilde-om3}}}{=}\|f\|_{\Om_\grpord^n}\,,
	\end{align}
	with the inequality holding because we are optimizing over a subset of points.
	From \eqref{eq:g-6} and \eqref{eq:g-3} we see Theorem \ref{thm: remez} would follow if we could prove
	\begin{equation*}
		|g(\sqrt{\omega},\sqrt{\omega},\ldots, \sqrt{\omega})|\Lesssim{d,\grpord}\|g\|_{\Om_\grpord^m}\,,
	\end{equation*}
	independent of $m\ge 1$. This is precisely Proposition \ref{lem:sqrtom-boundedness}.
\end{proof}

The proof of Proposition \ref{lem:sqrtom-boundedness} is the subject of the rest of this subsection. 
Our approach is to split $f$ into parts $f=\sum_{j}g_j$ such that each part $g_j$ has the properties A and B:

\begin{equation}\label{ineq:proof ideas}
	\|f\|_{\Om_\grpord^n}
	\overset{\text{Property A}}{\Gtrsim{d,\grpord}}
	\|g_j\|_{\Om_\grpord^n}
	\overset{\text{Property B}}{\Gtrsim{d,\grpord}}
	|g_j(\sqrt{\om})|\,.
\end{equation}

Such splitting gives
\[|f(\sqrt{\om})|\leq\sum_j|g_j(\sqrt{\om})|\Lesssim{d,\grpord}\sum_j\|g_j\|_{\Om_\grpord^n}\Lesssim{d,\grpord}\sum_j\|f\|_{\Om_\grpord^n}\,.\]
So as long as the number of $g_j$'s is independent of $n$ such a splitting with Properties A and B entails the result.

We will split $f$ via an operator that was first employed to prove the Bohnenblust--Hille inequality for cyclic groups \cite{SVZbh}.
We will only need the basic version of the operator here; a generalized version is considered in \cite{SVZbh}. Recall that any polynomial $f:\Omega_\grpord^n\to \C$ has the Fourier expansion 
\begin{equation*}
	f(z)=\sum_{\alpha\in\{0,1,\ldots, \grpord-1\}^n}a_\alpha z^\alpha.
\end{equation*}
Recall the support of a monomial $z^\alpha$ is $\supp(\alpha):=\{j:\alpha_j\neq 0\}$, and the support size $|\supp(\alpha)|$ refers to the cardinality of $\supp(\alpha)$.

\begin{defn}[Maximum support pseudoprojection]
	For any multi-index $\alpha\in\{0,1,\ldots, \grpord-1\}^n$ define the factor
	\begin{equation*}
		\tau_\alpha =\prod_{j: \alpha_j\neq 0}(1-\omega^{\alpha_j})\,.
	\end{equation*}
	For any polynomial on $\Om_\grpord^n$ with the largest support size $\ell\ge 0$
	\[f(z)=\sum_{|\supp(\alpha)|\le \ell}a_\alpha z^\alpha,\]
	we define $\mathfrak{D}f:\Om_\grpord^n\to \C$ via
	\[\mathfrak{D}f(z) =\sum_{|\supp(\alpha)|= \ell}\tau_\alpha\,a_\alpha z^\alpha\,.\]
\end{defn}

The operator $\mathfrak{D}$ can be considered as a Fourier multiplier, and this somewhat technical definition is motivated by the following key property,  the $L^\infty\to L^\infty$ boundedness when restricted to certain polynomials.

\begin{lemma}[Boundedness of maximum support pseudoprojection]
	\label{lem:d-bdd}
	Let $f:\Om_\grpord^n\to \mathbf{C}$ be a polynomial and $\ell$ be the maximum support size of monomials in $f$. Then
	\begin{equation}
		\label{ineq:main ingredient}
		\|\mathfrak{D}f\|_{\Om_\grpord^n}\le (2+2\sqrt{2})^{\ell} \|f\|_{\Om_\grpord^n}.
	\end{equation}
\end{lemma}
The proof of Lemma \ref{lem:d-bdd} is given in \cite{SVZbh}.
We repeat it here in a slightly simplified form for convenience.

\begin{proof}
	Let $\omega=e^{\frac{2\pi \iu }{K}}$. Consider the operator $G$:
	\begin{equation*}
		G(f)(x)
		=f\left(\frac{1+\omega}{2}+\frac{1-\omega}{2}x_1,\dots, \frac{1+\omega}{2}+\frac{1-\omega}{2}x_n\right),\qquad x\in\Om_2^n
	\end{equation*}
	that maps any function $f:\{1,\omega\}^n\subset \Om_\grpord^n\to \mathbf{C}$ to a function $G(f):\Om_2^n\to \mathbf{C}$.
	Then by definition
	\begin{equation}\label{ineq:G compare K prime}
		\|f\|_{\Om_\grpord^n}\ge 	\|f\|_{\{1,\omega\}^n}=\|G(f)\|_{\Om_2^n}. 
	\end{equation}
	Fix $m\le \ell$. For any $\alpha$ we denote 
	\begin{equation*}
		m_k(\alpha):=|\{j:\alpha_j=k\}|,\qquad 0\le k\le \grpord-1.
	\end{equation*}
	Then for $\alpha$ with $|\supp(\alpha)|=m$, we have
	\begin{equation*}
		m_1(\alpha)+\cdots+m_{\grpord-1}(\alpha)=|\supp(\alpha)|=m.
	\end{equation*}
	For $z\in \{1,\omega\}^n$ with $z_j=\frac{1+\omega}{2}+\frac{1-\omega}{2}x_j, x_j=\pm 1$, note that 
	\begin{equation*}
		z_j^{\alpha_j}=\left(\frac{1+\omega}{2}+\frac{1-\omega}{2}x_j \right)^{\alpha_j}=\frac{1+\omega^{\alpha_j}}{2}+\frac{1-\omega^{\alpha_j}}{2}x_j\,.
	\end{equation*}
	So for any $A\subset [n]$ with $|A|=m$, and for each $\alpha$ with $\supp(\alpha)=A$, we have for $z\in \{1,\omega\}^n$:
	\begin{align*}
		z^{\alpha}
		=&\prod_{j:\alpha_j\neq 0}z_j^{\alpha_j}\\
		=&\prod_{j:\alpha_j\neq 0}\left(\frac{1+\omega^{\alpha_j}}{2}+\frac{1-\omega^{\alpha_j}}{2}x_j\right)\\
		=&\prod_{j:\alpha_j\neq 0}\left(\frac{1-\omega^{\alpha_j}}{2}\right)\cdot x^A+\cdots\\
		=&2^{-m}\tau_\alpha x^A+\cdots
	\end{align*}
	where $x^A:=\prod_{j\in A}x_j$ is of degree $|A|=m$ while $\cdots$ is of degree $<m$. Then for $f(z)=\sum_{|\supp(\alpha)|\le \ell}a_{\alpha}z^\alpha$ we have
	\begin{align*}
		G(f)(x)
		=\sum_{m\le \ell}\frac{1}{2^m}\sum_{|A|=m}\left(\sum_{ \supp(\alpha)=A} \tau_\alpha a_{\alpha}\right)
		x^A+\cdots,\qquad x\in \Om_2^n.
	\end{align*}
	Again, for each $m\le \ell$, $\cdots$ is some polynomial of degree $<m$. So $G(f)$ is of degree $\le \ell$ and the $\ell$-homogeneous part is nothing but 
	\begin{equation*}
		\frac{1}{2^\ell}\sum_{|A|=\ell}\left(\sum_{ \supp(\alpha)=A} \tau_\alpha a_{\alpha}\right)
		x^A.
	\end{equation*}
	Consider the projection operator $Q$ that maps any polynomial on $\Om_2^n$ onto its highest level homogeneous part; \emph{i.e.}, for any polynomial $g:\Om_2^n\to \mathbf{C}$ with $\deg(g)=m$ we denote $Q(g)$ its $m$-homogeneous part. Then we just showed that 
	\begin{align}
		\label{eq:expl-form}
		Q(G(f))(x)	
		=\frac{1}{2^\ell}\sum_{|A|=\ell}\left(\sum_{ \supp(\alpha)=A} \tau_\alpha a_{\alpha}\right)
		x^A.
	\end{align}

	It is known that \cite[Lemma 1 (\textit{iv})]{DMP} for any polynomial $g:\Om^n_2\to \mathbf{C}$ of degree at most $d>0$ and $g_m$ its $m$-homogeneous part, $m\leq d$, we have the estimate
	\begin{equation*}
		\|g_m\|_{\Om_2^n}\le (1+\sqrt{2})^d \|g\|_{\Om_2^n}.
	\end{equation*}
	Applying this estimate to $G(f)$ and combining the result with \eqref{ineq:G compare K prime}, we have
\begin{equation*}
	\|Q(G(f))\|_{\Om_2^n}\le (\sqrt{2}+1)^{\ell}\|G(f)\|_{\Om_2^n}
	\le (1+\sqrt{2})^{\ell} \|f\|_{\Om_\grpord^n}
\end{equation*}
and thus by \eqref{eq:expl-form}
\begin{align*}
	\left\| \sum_{|A|=\ell}\left(\sum_{ \supp(\alpha)=A} \tau_\alpha  a_{\alpha}\right)
	x^A\right\|_{\Om_2^n}
	\le(2+2\sqrt{2})^{\ell}\|f\|_{\Om_\grpord^n}.
\end{align*}

The function on the left-hand side is almost $\mathfrak{D}f$.
Observe that $\Om_\grpord^n$ is a group, so we have 
\begin{equation*}
	\sup_{z,\xi\in\Om_\grpord^n}\left|\sum_{\alpha}a_{\alpha}z^{\alpha}\xi^{\alpha}\right|
	=\sup_{z\in\Om_\grpord^n}\left|\sum_{\alpha}a_{\alpha}z^{\alpha}\right|.
\end{equation*}
Thus we have actually shown
\begin{equation*}
	\sup_{z\in\Om_\grpord^n,x\in \Om_2^n}\left| \sum_{|A|=\ell}\left(\sum_{ \supp(\alpha)=A} \tau_\alpha a_{\alpha}z^{\alpha}\right)
	x^A\right|	\le (2+2\sqrt{2})^{\ell} \|f\|_{\Om_\grpord^n}.
\end{equation*}
Setting $x=\vec1$ gives \eqref{ineq:main ingredient}.
\end{proof}

Note that $\mathfrak{D}f$ is exactly the part of $f$ composed of monomials of maximum support size, except where the coefficients $a_\alpha$ have picked up the factor $\tau_\alpha$.
The relationships among the $\tau_\alpha$'s can be intricate: while in general they are different for distinct $\alpha$'s, this is not always true.
Consider the case of $\grpord=3$ and the two monomials
\begin{equation*}
z^{\beta}:=z_1^2 z_2z_3z_4z_5z_6z_7z_8,\quad z^{\beta'}:=z_1^2 z_2^2 z_3^2 z_4^2 z_5^2 z_6^2 z_7^2z_8\,.
\end{equation*}
Then
\[
\tau_{\beta}\;=\;(1-\omega)^7(1-\omega^2)\;=\;(1-\omega)(1-\omega^2)^7\;=\;\tau_{\beta'},
\]
which follows from the identity $(1-\omega)^6=(1-\omega^2)^6$ for $\omega=e^{2\pi \iu/3}$.

Understanding precisely when $\tau_\alpha = \tau_\beta$ seems to be a formidable task in transcendental number theory.
When $K$ is prime there is a relatively simple characterization (see Section \ref{sec:aside}) but for composite $K$ the situation is much less clear.
Nevertheless, it turns out that for the purposes of Theorem \ref{thm: remez} we do not need a full understanding.
Indeed, our $g_j$'s shall be defined according to the $\tau$'s.

\begin{defn}
Two monomials $z^\alpha,z^\beta$ with associated multi-indexes $\alpha,\beta\in \{0,1\dots, K-1\}^n$ are called \emph{inseparable} if $|\mathrm{supp}(\alpha)|=|\mathrm{supp}(\beta)|$ and $\tau_\alpha=\tau_\beta$.
When $m$ and $m'$ are inseparable, we write $m\sim m'$.

Inseparability is an equivalence relation among monomials. We may split any polynomial $f$ into parts $f=\sum_j g_j$ according to this relation. That is, any two monomials in $f$ are inseparable if and only if they belong to the same $g_j$. 
Call these $g_j$'s the \emph{inseparable parts} of $f$.
\end{defn}

It is these inseparable parts that are our $g_j$'s in \eqref{ineq:proof ideas}.
We shall formally check it later, but it is easy to see the number of inseparable parts is independent of $n$.
We formulate and prove Properties A \& B next.

\subsubsection{Property A: Boundedness of inseparable parts}
\label{sec:boundedness}

Repeated applications of the operator $\mathfrak{D}$ enable splitting into inseparable parts.

\begin{prop}[Property A]\label{prop:Property A}
Fix $\grpord\ge 3$ and $d\ge 1$. Suppose that $f:\Om_\grpord^n\to \mathbf{C}$ is a polynomial of degree at most $d$ with maximum support size $L$. For $0\leq\ell\leq L$ let $f_\ell$ denote the part of $f$ composed of monomials of support size $\ell$, and let $g_{(\ell, 1)},\ldots, g_{(\ell, J_\ell)}$ be the inseparable parts of $f_\ell$.
Then there exists a universal constant $C_{d,\grpord}$ independent of $n$ and $f$ such that for all $0\leq \ell\leq L$ and $1\leq j\leq J_\ell$,
\[\|g_{(\ell, j)}\|_{\Om_\grpord^n}\leq C_{d,\grpord}\|f\|_{\Om_\grpord^n}\,.\]
\end{prop}
\begin{proof}
We first show the proposition for $g_{(L,j)}$, $1\leq j \leq J_L$. Suppose that 
\begin{equation*}
	f(z)=\sum_{\alpha:|\supp(\alpha)|\le L}a_{\alpha}z^{\alpha}.
\end{equation*}
Inductively, one obtains from Lemma \ref{lem:d-bdd} that for $1\leq k\leq J_L$,
\begin{align}
	\label{ineq:c(alpha)^k}
	\begin{split}
		\mathfrak{D}^{k}f&= \sum_{ |\supp(\alpha)|=L} \tau_\alpha^ka_{\alpha} z^{\alpha}
		\\
		\text{with}\qquad\left\| \mathfrak{D}^{k}f\right\|_{\Om_\grpord^n}&\le (2+2\sqrt{2})^{kL}\|f\|_{\Om_\grpord^n}.
	\end{split}
\end{align}
By definition there are $J_L$ distinct values of $\tau_\alpha$ among 
the monomials of $f_L$; label them $c_1,\ldots, c_{J_L}$.
Then
\begin{equation*}
	f_{L}(z)=\sum_{ |\supp(\alpha)|=L} a_{\alpha} z^{\alpha}=\sum_{1\le j\le J_L}g_{(L,j)}(z), \quad\text{and}
\end{equation*}
\begin{equation*}
	\mathfrak{D}^{k}f(z)=\sum_{ |\supp(\alpha)|=L} \tau_\alpha^k a_{\alpha}z^{\alpha}=\sum_{1\le j\le J_L}c_j^k g_{(L,j)}(z),\qquad k\ge 1.
\end{equation*}
Let us confirm $J_L$ is independent of $n$.
Consider $\alpha$ with $|\supp(\alpha)|=L$.
We may count the support size of $\alpha$ by binning coordinates according to their degree:
$|\supp(\alpha)|=L$,
\[\sum_{1\leq t\leq \grpord-1}|\{s\in[n]:\alpha_s=t\}|=L\leq d,\]
so 
\begin{align}
\label{ineq:J_L}
	\begin{split}
	J_L &\leq |\{(m_1,\dots, m_{\grpord-1})\in \{0,\ldots, L\}^{\grpord-1}:m_1+\cdots  +m_{\grpord-1}= L\}|\\
	 &\le  \binom{\grpord-1+L-1}{L-1}\leq (\grpord+d)^d\,.
	\end{split}
\end{align}

According to \eqref{ineq:c(alpha)^k}, we have

\begin{equation*}
	\begin{pmatrix}
		\mathfrak{D}f\\
		\mathfrak{D}^{2}\!f\\
		\vdots \\
		\mathfrak{D}^{J_L}\!f\\
	\end{pmatrix}
	=\underbracedmatrix{\begin{matrix}
			c_1 & c_2 & \cdots & c_{J_L}\\
			c_1^2 & c_2^2 &\cdots & c_{J_L}^2\\
			\vdots & \vdots & \ddots & \vdots\\
			c_1^{J_L} & c_2^{J_L} & \cdots & c_{J_L}^{J_L}
	\end{matrix}}{=:\; V_L}
	\begin{pmatrix}
		g_{(L,1)}\\
		g_{(L,2)}\\
		\vdots \\
		g_{(L,J_L)}\\
	\end{pmatrix}.
\end{equation*}

The $J_L\times J_L$ modified Vandermonde matrix $V_L$ has determinant
\[\det(V_L) = \left(\prod_{j=1}^{J_L} c_j\right) \left( \prod_{1\leq s < t \leq J_L}(c_s-c_t)\right).\]
Since the $c_j$'s are distinct and nonzero we have $\det(V_L)\neq 0$. So $V_L$ is invertible and in particular $g_{(L,j)}$ is the $j$\textsuperscript{th} entry of $V_L^{-1}(\mathfrak{D}^1\!f,\ldots,\mathfrak{D}^{J_L}\!f)^\top$.
Letting $\eta^{(L,j)}=(\eta^{(L,j)}_k)_{1\le k\le J_L}$ be the $j$\textsuperscript{th} row of $V_L^{-1}$, this means
\begin{align*}
	g_{(L,j)}&=\sum_{1\le k\le J_L}\eta_k^{(L,j)}\, \mathfrak{D}^{k}\!f.
\end{align*}
As $\eta^{(L,j)}$ depends on $d$ and $\grpord$ only, so for all $1\leq j\leq J_L$,
\begin{equation}
	\label{eq:top-supp-bd}
	\|g_{(L,j)}\|_{\Om_\grpord^n}\le \sum_{1\le k\le J_L} \big|\eta_k^{(L,j)}\big|\!\cdot\!	\left\|\mathfrak{D}^{k}\!f\right\|_{\Om_\grpord^n}
	\le \|\eta^{(L,j)}\|_1\big(2+2\sqrt{2}\big)^{J_Ld}	\|f\|_{\Om_\grpord^n},
\end{equation}
where we used \eqref{ineq:c(alpha)^k} in the last inequality. The constant \[\|\eta^{(L,j)}\|_1(2+2\sqrt{2})^{J_Ld}\leq C(d,\grpord)<\infty\]
for appropriate $C(d,\grpord)$ that is dimension-free and depends only on $d$ and $\grpord$ only.
This finishes the proof for the inseparable parts in $f_L$.

We now repeat the argument on $f-f_L$ to obtain \eqref{eq:top-supp-bd} for the inseparable parts of support size $L-1$.
In particular, there are vectors $\eta^{(L-1,j)}$, $1\leq j\leq J_{L-1}$ of dimension-free 1-norm with
\[\|g_{(L-1,j)}\|_{\Om_\grpord^n}\leq C(d,\grpord)\|\eta^{(L-1,j)}\|_1\|f-f_L\|_{\Om_\grpord^n}
\Lesssim{d,\grpord}\|f-f_L\|_{\Om_\grpord^n}\,.\]
This can be further repeated to obtain for $0\leq \ell\leq L$ and $1\leq j\leq J_\ell$, the vectors $\eta^{(\ell,j)}$ with dimension-free 1-norm such that
\[\|g_{(\ell, j)}\|_{\Om_\grpord^n}\Lesssim{d,\grpord}\left\|f-\sum_{\ell+1\leq k\leq L}f_k\right\|_{\Om_\grpord^n}\,.\]

It remains to relate $\|f-\sum_{\ell+1\leq k \leq L}f_k\|_{\Om_\grpord^n}$ to $\|f\|_{\Om_\grpord^n}$.
Note that with $V_L$ as originally defined, by considering $(1\,1\,\dots\,1)V_L^{-1}(\mathfrak{D}^1f,\ldots, \mathfrak{D}^{J_L}f)^\top$ we obtain a constant $D_L=D_L(d,\grpord)$ independent of $n$ for which
\[\|f_L\|_{\Om_\grpord^n}\leq D_L\|f\|_{\Om_\grpord^n}\,.\]
This means
\[\|f-f_L\|_{\Om_\grpord^n}\leq (1+D_L)\|f\|_{\Om_\grpord^n}\,.\]
Notice the top-support part of $f-f_L$ is exactly $f_{L-1}$, so repeating the argument above on $f-f_L$ yields a constant $D_{L-1}=D_{L-1}(d,\grpord)$ such that
\begin{equation*}
    \|f_{L-1}\|_{\Om_\grpord^n}\leq D_{L-1}\|f-f_L\|_{\Om_\grpord^n}
\le D_{L-1}(1+D_L)\|f\|_{\Om_\grpord^n}=(D_{L-1}+D_{L-1}D_L)\|f\|_{\Om_\grpord^n}\,.
\end{equation*}
        Continuing, for $1\leq \ell \leq L$ we find 
\begin{align*}
	\|f_{L-\ell}\|_{\Om_\grpord^n}
	&\leq D_{L-\ell}\|f-\sum_{L-\ell+1\le k\le L}f_k\|_{\Om_\grpord^n}\\
	&\leq D_{L-\ell}(1+D_{L-\ell+1})\|f-\sum_{L-\ell+2\le k\le  L}f_k\|_{\Om_\grpord^n}\\
	&\leq D_{L-\ell}\prod_{0\le k\le \ell-1}(1+D_{L-k})\|f\|_{\Om_\grpord^n}.
\end{align*}
We have found for each $\ell$-support-homogeneous part of $f$,
\[\|f_\ell\|_{\Om_\grpord^n}\Lesssim{d,\grpord}\|f\|_{\Om_\grpord^n},\]
so we have $\|f-\sum_{\ell+1\leq k \leq L}f_k\|_{\Om_\grpord^n}\Lesssim{d,\grpord}\|f\|_{\Om_\grpord^n}$ as well.
\end{proof}

\subsubsection{Property B: Boundedness at \texorpdfstring{$\sqrt{\omega}$}{√ω} for inseparable parts}

Here we argue $g(\sqrt{\omega})$ is bounded for inseparable $g$. Recall that $\omega=e^{\frac{2\pi \iu}{K}}$ and $\sqrt{\omega}=e^{\frac{\pi \iu}{K}}$.
\begin{prop}[Property B]
\label{prop:Property B}
If $g$ is inseparable then $|g(\sqrt{\omega})|\leq\|g\|_{\Om_\grpord^n}$.
\end{prop}
\begin{proof}
We will need an identity for half-roots of unity.
For $k=1,\ldots, \grpord-1$ we have
\begin{equation}
	\label{eq:trig-id}
	(\sqrt{\omega})^k=\iu\frac{1-\omega^k}{|1-\omega^k|}\,,
\end{equation}
following from the orthogonality of $(\sqrt{\omega})^k$ and $1-\omega^k$ in the complex plane.

We claim that for two monomials $m$ and $m'$
\begin{equation*}
	m\sim m' \;\implies\; m(\sqrt{\om})=m'(\sqrt{\om})\,.
\end{equation*}
By definition $m\sim m'$ means $m$ and $m'$ have the same support size (call it $\ell$) and
\[\textstyle\prod_{j:\alpha_j\neq0}(1-\omega^{\alpha_j}) = \prod_{j:\beta_j\neq0}(1-\omega^{\beta_j})\,.\]
Dividing both sides by the modulus and multiplying by $\iu^\ell$ allows us to apply \eqref{eq:trig-id} to find
\[\textstyle\prod_{j:\alpha_j\neq0}(\sqrt{\omega})^{\alpha_j}=\prod_{j:\beta_j\neq0}(\sqrt{\omega})^{\beta_j}\,,\]
as desired.

Now let $\zeta=m(\sqrt{\omega})\in \T$ for some monomial $m$ in $g$.
Then because $\zeta$ is independent of $m$, with $g=\sum_{\alpha\in S}a_\alpha z^\alpha$, we have $g(\sqrt{\omega})=\zeta\sum_{\alpha\in S}a_\alpha$ and
\begin{equation*}
	\textstyle
	|g(\sqrt{\omega})| =|\sum_{\alpha\in S}a_\alpha| = |g(\vec1)|\leq \|g\|_{\Om_\grpord^n}\,.\qedhere
\end{equation*}
\end{proof}

We may now prove Proposition \ref{lem:sqrtom-boundedness}.
\begin{proof}[Proof of Proposition \ref{lem:sqrtom-boundedness}.]
Write $f=\sum_{0\leq \ell \leq L}\sum_{1\leq j\leq J_\ell}g_{(\ell,j)}$ in terms of inseparable parts, where $g_{(\ell,j)},1\le j\le J_\ell,0\le \ell \le L$ are as in Proposition \ref{prop:Property A}. Then by Propositions \ref{prop:Property A} (Property A) and \ref{prop:Property B} (Property B)
\begin{align*}
	|f(\sqrt{\omega})|\;&\leq\quad\sum_{0\leq \ell \leq L}\sum_{1\leq j\leq J_\ell}|g_{(\ell,j)}(\sqrt{\omega})|\\
	&\leq\quad\sum_{0\leq \ell \leq L}\sum_{1\leq j\leq J_\ell}\|g_{(\ell,j)}\|_{\Om_\grpord^n}\tag{Property B}\\
	&\Lesssim{d,\grpord}\|f\|_{\Om_\grpord^n}\sum_{0\leq \ell \leq L}J_\ell\tag{Property A}\,.
\end{align*}
In view of \eqref{ineq:J_L} and $L\le d$, we obtain $|f(\sqrt{\omega})|\Lesssim{d,\grpord}\|f\|_{\Om_\grpord^n}$.
\end{proof}

\subsection{Aside: characterizing inseparable parts for prime \texorpdfstring{$\grpord$}{K}}
\label{sec:aside}

Although it is not required for the proof of Theorem \ref{thm: remez}, it is interesting to understand what are the parts $g$ of $f$ for which
\begin{equation}
	\label{ineq:prime-proj}
	\|g\|_{\Om_\grpord^n}\Lesssim{d,\grpord}\|f\|_{\Om_\grpord^n}
\end{equation}
via our Property A (Proposition \ref{prop:Property A})?
Recall that \eqref{ineq:prime-proj} holds when $g$ is a part of $f$ containing all monomials in $f$ from an equivalence class of the inseparability equivalence relation $\sim$.

Thus we are led to ask for a characterization of inseparability.
It turns out that for prime $\grpord$ this can be done completely via connections to transcendental number theory including Baker's theorem \cite{Baker}.

\begin{prop}
	Suppose $\grpord\ge 3$ is prime and $\alpha,\beta\in{\{0,1,\ldots, K-1\}}^n$.
	Then two monomials $z^\alpha$, $z^\beta$ are inseparable if and only if 
	\begin{itemize}
		\item Support sizes are equal: $|\supp(\alpha)|=|\supp(\beta)|$,
		\item Degrees are equal mod $2\grpord$: $|\alpha|=|\beta| \mod 2\grpord$,
		\item Individual degree symmetry: there is a bijection $\pi:\supp(\alpha)\to\supp(\beta)$ such that for all $j\in\supp(\alpha)$, $\alpha_j=\beta_{\pi(j)}$ or $\alpha_j = \grpord-\beta_{\pi(j)}$.
	\end{itemize}
\end{prop}
\begin{proof}
	Recall that by definition, two monomials $z^\alpha$ and $z^\beta$ are inseparable if and only if they have the same support size and $\tau_\alpha=\tau_\beta$; that is,
	\[\prod_{j:\alpha_j\neq0}(1-\omega^{\alpha_j})=\prod_{j:\beta_j\neq0} (1-\omega^{\beta_j})\,,\]
	where $\omega = e^{2\pi \iu/\grpord}$.
	For these quantities to be equal, their respective moduli and arguments must coincide.
	
	To compare arguments, observe that for any multi-index $\sigma \in\{0,1,\ldots,\grpord-1\}^n$, by the identity \eqref{eq:trig-id} we may normalize $\tau_\sigma$ like so:
	\[\frac{\tau_\sigma}{|\tau_\sigma|}=\iu^{-|\supp(\sigma)|}\prod_{j=1}^n(\sqrt{\omega})^{\sigma_j}=\iu^{-|\supp(\sigma)|}(\sqrt{\omega})^{|\sigma|},\]
	where as before $\sqrt{\omega}=e^{\pi\iu/K}$.
	It is given that $|\supp(\alpha)|=|\supp(\beta)|$, so the arguments of $\tau_\alpha$ and $\tau_\beta$ are equal exactly when $(\sqrt{\omega})^{|\alpha|}=(\sqrt{\omega})^{|\beta|}$, or equivalently, $|\alpha| =|\beta| \text{ mod } 2\grpord$.
	
	As for the moduli, using the identity $|1-\omega^k| = 2\sin(k \pi/\grpord)$ we find for any multi-index $\sigma$ that
	\[|\tau_\sigma|=\prod_{j:\sigma_j\neq 0}2\sin(\sigma_j \pi / \grpord)=\prod_{j:\sigma_j\neq 0}2\sin(\min\{\sigma_j,\grpord-\sigma_j\}\cdot\pi / K)\,,\]
	where the last step follows from the symmetry of sine about $\pi/2$.
	
	So when are $|\tau_\alpha|$ and $|\tau_\beta|$ equal?
	By the last display, certainly they are the same if there is a bijection $\pi:\supp(\alpha)\to\supp(\beta)$ such that for all $j\in\supp(\alpha)$, $\alpha_j=\beta_{\pi(j)}$ or $\alpha_j = \grpord-\beta_{\pi(j)}$.
	Is this the only time $|\tau_\alpha|=|\tau_\beta|$?
	
	Returning to $\sigma$, define for $1\leq k \leq (\grpord-1)/2$ the quantity
	\[\widehat{\sigma}(k)=|\{j:\sigma_j=k \text{ or } \sigma_j=\grpord-k\}|\,.\]
	Then
	\[\log(|\tau_\sigma|)=\sum_{k=1}^{(\grpord-1)/2}\widehat{\sigma}(k)\cdot\log(2 \sin(k \pi/\grpord)).\]
	Therefore if the numbers
	\[\{b_k:=\log(2 \sin(k \pi/\grpord)), k=1,\ldots,(\grpord-1)/2\}\]
		 were linearly independent over $\Z$, the only way $|\tau_\alpha|=|\tau_\beta|$ is the existence of a bijection $\pi$ as above.
		 
	Conveniently, the question of the linear independence of the $b_k$'s has already appeared in a different context, concerning an approach of Livingston to resolve a folklore conjecture of Erd\" os on the vanishing of certain Dirichlet $L$-series.
	It was answered in \cite{P} in the positive for $\grpord\geq 3$ prime and in the negative for all composite $\grpord\geq 4$ using several tools including Baker's celebrated theorem on linear forms in logarithms of algebraic numbers \cite{Baker}.
\end{proof}

Finally, recalling \eqref{ineq:generalized cauchy} that Theorem \ref{thm: remez} implies $\|f_k\|_{\Om_\grpord^n}\Lesssim{d,K}\|f\|_{\Om_\grpord^n}$ for all $k$-homogeneous parts $f_k$ of $f$, $0\leq k \leq d$, we may conclude by Proposition \ref{prop:Property A}:
\begin{corollary}
	\label{cor:bounded-parts}
	Suppose $\grpord$ is an odd prime and let $S$ be a maximal subset of $\{0,1,\ldots, \grpord-1\}^n$ such that for all $\alpha, \beta\in S$:
	\begin{itemize}
		\item Support sizes are equal: $|\supp(\alpha)|=|\supp(\beta)|$.
		\item Degrees are equal: $|\alpha|=|\beta|$.
		\item Individual degree symmetry: there is a bijection $\pi:\supp(\alpha)\to\supp(\beta)$ such that for all $j\in\supp(\alpha)$, $\alpha_j=\beta_{\pi(j)}$ or $\alpha_j = \grpord-\beta_{\pi(j)}$.
	\end{itemize}
	Then for any $n$-variate analytic polynomial $f$ of degree at most $d$ and individual degree at most $K-1$, the $S$-part of $f$, \emph{i.e.,}
	$f_S=\sum_{\alpha\in S}\widehat{f}(\alpha)z^\alpha$,
	satisfies:
	\[\|f_S\|_{\Om_\grpord^n}\Lesssim{d,\grpord}\|f\|_{\Om_\grpord^n}\,.\]
	
\end{corollary}


\section*{Acknowledgments} 
A.V. is grateful to Nazarov for encouraging discussions. The authors are grateful to the anonymous reviewers for their helpful comments. This work was started while all three authors were in residence at the Institute for Computational and Experimental Research in Mathematics in Providence, RI, during the Harmonic Analysis and Convexity program. It is partially supported by NSF  DMS-1929284.

\bibliographystyle{amsplain}

\begin{thebibliography}{99}

\bibitem{ABE}
Richard Aron, Bernard Beauzamy, and Per Enflo, \emph{Polynomials in many
  variables: real vs complex norms}, J. Approx. Theory \textbf{74} (1993),
  no.~2, 181--198.
  
  
\bibitem{Baker}
Alan Baker, \emph{Transcendental number theory}, Cambridge Mathematical
  Library, Cambridge University Press, Cambridge, 2022, With an introduction by
  David Masser, Reprint of the 1975 original [0422171]. 

\bibitem{bernstein31}
S.~Bernstein, \emph{Sur une classe de formules d'interpolation.}, Bull. Acad.
  Sci. URSS \textbf{1931} (1931), no.~9, 1151--1161 (French).

\bibitem{bernstein32}
Serge Bernstein, \emph{Sur une modification de la formule d'interpolation de
  {Lagrange}}, Commun. {Soc}. {Math}. {Kharkow} et {Inst}. {Sci}. {Math}.
  {Ukraine}, {IV}. {Ser}. 5, 49-57 (1932)., 1932.

  
\bibitem{BKSVZ}
Lars Becker, Ohad Klein, Joseph Slote, Alexander Volberg, and Haonan Zhang,
  \emph{Dimension-free discretizations of the uniform norm by small product
  sets}, Invent. Math. (2024).


\bibitem{blei}
Ron Blei, \emph{Analysis in integer and fractional dimensions}, Cambridge
  Studies in Advanced Mathematics, vol.~71, Cambridge University Press,
  Cambridge, 2001. 

  
\bibitem{BH}
H.~F. Bohnenblust and Einar Hille, \emph{On the absolute convergence of
  {D}irichlet series}, Ann. of Math. (2) \textbf{32} (1931), no.~3, 600--622.
  

\bibitem{BY}
A.~Brudnyi and Y.~Yomdin, \emph{Norming sets and related {R}emez-type
  inequalities}, J. Aust. Math. Soc. \textbf{100} (2016), no.~2, 163--181.

 
\bibitem{BG}
Ju. ~A. Brudnyi and M.~I. Ganzburg, \emph{A certain extremal problem for
  polynomials in {$n$} variables}, Izv. Akad. Nauk SSSR Ser. Mat. \textbf{37}
  (1973), 344--355. 
  

\bibitem{DMP}
Andreas Defant, Mieczys\l aw Masty\l o, and Antonio P\'erez, \emph{On the
  {F}ourier spectrum of functions on {B}oolean cubes}, Math. Ann. \textbf{374}
  (2019), no.~1-2, 653--680. 


\bibitem{DPTT19}
F.~Dai, A.~Primak, V.~N. Temlyakov, and S.~Yu. Tikhonov, \emph{Integral norm
  discretization and related problems}, Uspekhi Mat. Nauk \textbf{74} (2019),
  no.~4(448), 3--58. 

  \bibitem{FM}
Natacha Fontes-Merz, \emph{A multidimensional version of {T}ur\'an's lemma}, J.
  Approx. Theory \textbf{140} (2006), no.~1, 27--30. 
  

\bibitem{Fr}
Matthieu Fradelizi, \emph{Concentration inequalities for {$s$}-concave measures
  of dilations of {B}orel sets and applications}, Electron. J. Probab.
  \textbf{14} (2009), no. 71, 2068--2090. 
  

\bibitem{FY}
Omer Friedland and Yosef Yomdin, \emph{An observation on the
  {T}ur\'an-{N}azarov inequality}, Studia Math. \textbf{218} (2013), no.~1,
  27--39. 
  

\bibitem{G}
Michael~I. Ganzburg, \emph{A multivariate {R}emez-type inequality with
  {$\varphi$}-concave weights}, Colloq. Math. \textbf{147} (2017), no.~2,
  221--240. 


\bibitem{K}
Maciej Klimek, \emph{Metrics associated with extremal plurisubharmonic
  functions}, Proc. Amer. Math. Soc. \textbf{123} (1995), no.~9, 2763--2770.
  
\bibitem{KKLT22}
B.~Kashin, E.~Kosov, I.~Limonova, and V.~Temlyakov, \emph{Sampling
  discretization and related problems}, J. Complexity \textbf{71} (2022), Paper
  No. 101653, 55.

\bibitem{KS}
Andr\'as Kro\'o and Darrell Schmidt, \emph{Some extremal problems for
  multivariate polynomials on convex bodies}, J. Approx. Theory \textbf{90}
  (1997), no.~3, 415--434. 


\bibitem{L}
Magnus Lundin, \emph{The extremal {PSH} for the complement of convex, symmetric
  subsets of {${\bf R}^N$}}, Michigan Math. J. \textbf{32} (1985), no.~2,
  197--201.


  \bibitem{Nazarov}
F.~L. Nazarov, \emph{Local estimates for exponential polynomials and their
  applications to inequalities of the uncertainty principle type}, Algebra i
  Analiz \textbf{5} (1993), no.~4, 3--66. 
  

\bibitem{NSV1}
F.~Nazarov, M.~Sodin, and A.~Volberg, \emph{Local dimension-free estimates for
  volumes of sublevel sets of analytic functions}, Israel J. Math. \textbf{133}
  (2003), 269--283. 
  
  
\bibitem{NSV}
F.~Nazarov, M.~Sodin, and A.~Volberg, \emph{The geometric
  {K}annan-{L}ov\'asz-{S}imonovits lemma, dimension-free estimates for the
  distribution of the values of polynomials, and the distribution of the zeros
  of random analytic functions}, Algebra i Analiz \textbf{14} (2002), no.~2,
  214--234. 
  

\bibitem{P}
Siddhi Pathak, \emph{On a conjecture of {L}ivingston}, Canad. Math. Bull.
  \textbf{60} (2017), no.~1, 184--195. 


\bibitem{R}
E.~J. Remez, \emph{Sur une propri{\'e}t{\'e} des polyn{\^o}mes de tchebycheff},
  Comm. Inst. Sci. Kharkov \textbf{13} (1936), 93--95.


\bibitem{SVZbh}
Joseph Slote, Alexander Volberg, and Haonan Zhang, \emph{Bohnenblust-{H}ille
  inequality for cyclic groups}, Adv. Math. \textbf{452} (2024), Paper No.
  109824, 35. 
  

\bibitem{SVZ}
Joseph Slote, Alexander Volberg, and Haonan Zhang, \emph{Noncommutative Bohnenblust--Hille inequality for qudit systems
},
  arXiv preprint (arXiv:2406.08509, 2024).

\bibitem{T}
Paul Tur\'an, \emph{Eine neue {M}ethode in der {A}nalysis und deren
  {A}nwendungen}, Akad\'emiai Kiad\'o, Budapest, 1953.
  
\bibitem{VZ22}
Alexander Volberg and Haonan Zhang, \emph{Noncommutative {B}ohnenblust-{H}ille
  inequalities}, Math. Ann. \textbf{389} (2024), no.~2, 1657--1676.
  
\bibitem{Y}
Y.~Yomdin, \emph{Remez-type inequality for discrete sets}, Israel J. Math. \textbf{186} (2011), 45--60. 

\bibitem{Zygmund}
Antoni Zygmund, \emph{Trigonometric {I}nterpolation}, University of Chicago,
  Chicago, IL, 1950.

\end{thebibliography}


\begin{dajauthors}
\begin{authorinfo}[joe]
  Joseph Slote\\
  Department of Computing and Mathematical Sciences, California Institute of Technology\\
  Pasadena, CA 91125\\
  USA\\
  jslote\imageat{}caltech\imagedot{}edu \\
  \url{https://joeslote.com/}
\end{authorinfo}

\begin{authorinfo}[sasha]
  Alexander Volberg\\
  Department of Mathematics, Michigan State University\\
East Lansing, MI 48823\\
USA\\
and \\
Hausdorff Center for Mathematics, University of Bonn\\
Bonn 53115,\\
Germany\\
volberg\imageat{}msu\imagedot{}edu \\
\url{https://users.math.msu.edu/users/volberg/}
\end{authorinfo}

\begin{authorinfo}[haonan]
  Haonan Zhang\\
  Department of Mathematics, University of South Carolina\\
  Columbia, SC 29208\\
  USA\\
  haonanzhangmath\imageat{}gmail\imagedot{}com\\
  \url{https://sites.google.com/view/haonanzhangmathsearch-console/home}
\end{authorinfo}

\end{dajauthors}

\end{document}